\documentclass[11pt]{amsart}

\usepackage{amsfonts}
\usepackage{amssymb}
\usepackage{amsmath}
\usepackage{amsthm}
\usepackage[colorlinks,allcolors=magenta]{hyperref}
\usepackage[all]{xy}
\usepackage{enumitem}

\setenumerate{topsep=0pt,labelwidth=0pt,align=left,labelindent=0pt,labelsep=0in,labelwidth=0.275in,leftmargin=0.275in,label=\rm(\alph*),parsep=0pt,itemsep=3pt}

\newtheorem{thm}{Theorem}[section]
\newtheorem{prop}[thm]{Proposition}
\newtheorem{lemma}[thm]{Lemma}
\newtheorem{cor}[thm]{Corollary}
\newtheorem{ques}{Question}

\theoremstyle{definition}
\newtheorem{defn}[thm]{Definition}

\theoremstyle{remark}

\newcommand{\N}{\mathbb{N}}

\newcommand{\R}{\mathbb{R}}

\newcommand{\B}{\mathbb{B}}
\newcommand{\A}{\mathbb{A}}

\newcommand{\LA}{\mathcal{A}}
\newcommand{\LB}{\mathcal{B}}

\newcommand{\LF}{\mathcal{F}}

\newcommand{\LU}{\mathcal{U}}
\newcommand{\LV}{\mathcal{V}}

\newcommand{\FIN}{\mathrm{FIN}}
\newcommand{\supp}{\mathrm{supp}}

\newcommand{\concat}{^\smallfrown}

\newcommand{\ZFC}{\mathsf{ZFC}}
\newcommand{\CH}{\mathsf{CH}}
\newcommand{\MA}{\mathsf{MA}}

\newcommand{\V}{\mathbf{V}}
\renewcommand{\L}{\mathbf{L}}

\newcommand{\restr}{\!\upharpoonright\!}

\newcommand{\osc}{\mathrm{osc}}

\newcommand{\D}{\mathbb{D}}

\title[Filters on a countable vector space]{Filters on a countable vector space}
\date{October 3, 2021}
\author{Iian B. Smythe}
\address{Department of Mathematics, University of Michigan, East Hall, 530 Church Street, Ann Arbor, MI 48109}
\urladdr{www.iiansmythe.com}
\email{smythe@umich.edu}

\subjclass[2010]{Primary 03E05, Secondary 15A03.}
\thanks{The author would like to thank Andreas Blass for many insightful conversations which contributed to this work.}

\begin{document}

\begin{abstract}
	We study various combinatorial properties, and the implications between them, for filters generated by infinite-dimensional subspaces of a countable vector space. These properties are analogous to selectivity for ultrafilters on the natural numbers and stability for ordered-union ultrafilters on $\FIN$.
\end{abstract}

\maketitle

\section{Introduction}

Throughout, we fix a countably infinite-dimensional vector space $E$ over a countable (possibly finite) field $F$, with distinguished basis $(e_n)$; one may take $E=\bigoplus_n F$ and $e_n$ the $n$th unit coordinate vector. We will use the term ``subspace'' in reference to $E$ exclusively to mean linear subspace. Our primary objects of study here are filters of subsets of $E\setminus\{0\}$ (we abuse terminology and call these filters on $E$), generated by infinite-dimensional subspaces. All such filters will be assumed to be proper and contain all subspaces of finite codimension.

We follow the terminology and notation of \cite{MR3864398}. A sequence $(x_n)$ of nonzero vectors in $E$ is called a \emph{block sequence} (and its span, a \emph{block subspace}) if for all $n$,
\[
	\max(\supp(x_{n}))<\min(\supp(x_{n+1})),
\]
where the \emph{support} of a nonzero vector $v$, $\supp(v)$, is the finite set of those $i$'s such that $e_i$ has a nonzero coefficient in the basis expansion of $v$. By taking linear combinations of basis vectors and thinning out, we can see that every infinite-dimensional subspace contains an infinite block sequence. Note that $\supp(v)$ is an element of $\FIN$, the set of nonempty finite subsets of $\omega$.

The set of infinite block sequences in $E$ is denoted by ${E^{[\infty]}}$ and inherits a Polish topology from $E^\omega$, where is $E$ discrete. We denote the set of finite block sequences by ${E^{[<\infty]}}$. Block sequences are ordered by their spans: we write $X\preceq Y$ if $\langle X\rangle\subseteq\langle Y\rangle$, where $\langle\cdot\rangle$ denotes the span (with $0$ removed), or equivalently, each entry of $X$ is a linear combination of entries from $Y$. We write $X/n$ (or $X/\vec{x}$, for $\vec{x}\in{E^{[<\infty]}}$) for the tail of $X$ consisting of those vectors with supports entirely above $n$ (or the supports of $\vec{x}$, respectively), and $X\preceq^* Y$ if $X/n\preceq Y$ for some $n$.

\begin{defn}
	A filter $\LF$ on $E$ is a \emph{block filter} if it has a base of sets of the form $\langle X\rangle$ for $X\in{E^{[\infty]}}$.	
\end{defn}

From now on, whenever we use the notation $\langle X\rangle$, it will be understood that $X\in{E^{[\infty]}}$.

In \cite{MR3864398}, we considered \emph{families}\footnote{Those readers familiar with \cite{MR3864398} should be cautioned that many of the definitions for families therein simplify in the case of filters, and that some of the results in the present article may no longer hold when ``filter'' is replaced by ``family''. This relationship is similar to that between ultrafilters and the more general class of coideals on $\omega$.} in ${E^{[\infty]}}$, i.e., subsets of ${E^{[\infty]}}$ which are upwards closed with respect to $\preceq^*$, with filters in $({E^{[\infty]}},\preceq)$ being those families in ${E^{[\infty]}}$ which are $\preceq$-downwards directed, as an important special case. As remarked there, one can go back and forth between filters in $({E^{[\infty]}},\preceq)$ and the block filters on $E$ they generate by taking spans and their inverse images, respectively.

\begin{defn}
	Given a block filter $\LF$ on $E$:
	\begin{enumerate}
		\item a set $D\subseteq E$ is \emph{$\LF$-dense} if for every $\langle X\rangle\in\LF$, there is an infinite-dimensional subspace $V\subseteq \langle X\rangle$ such that $V\subseteq D$.
		\item $\LF$ is \emph{full} if whenever $D\subseteq E$ is $\LF$-dense, we have that $D\in\LF$.
	\end{enumerate}
\end{defn}

Fullness is the analogue in this setting to being an ultrafilter: A filter $\LU$ on $\omega$ is an ultrafilter if and only if whenever $d\subseteq\omega$ has infinite intersection with each element of $\LU$, $d\in\LU$.

Already, the notion of a full block filter has substantial content: While they can be constructed using the Continuum Hypothesis $(\CH)$, Martin's Axiom $(\MA)$, or by forcing directly with $(E^{[\infty]},\preceq)$, they project to ordered-union ultrafilters on $\FIN$ via supports and thus cannot be proved to exist in $\ZFC$ alone (see \S5 and \S6 of \cite{MR3864398} for details). The following additional properties can also be obtained by the same methods:

\begin{defn}
 	A block filter $\LF$ on $E$ is:
 	\begin{enumerate}
 		\item a \emph{$(p)$-filter} (or has the \emph{$(p)$-property}) if whenever $\langle X_n\rangle\in\LF$ for all $n$, there is an $\langle X\rangle\in\LF$ such that $X\preceq^* X_n$ for all $n$.
 		\item \emph{spread} if whenever $I_0<I_1<I_2<\cdots$ is a sequence of intervals in $\omega$, there is an $\langle X\rangle\in\LF$, where $X=(x_n)$, such that for every $n$, there is an $m$ for which $I_0<\supp(x_n)<I_m<\supp(x_{n+1})$.
 		\item a \emph{strong $(p)$-filter} (or has the \emph{strong $(p)$-property}) if whenever $\langle X_{\vec{x}}\rangle\in\LF$ for all $\vec{x}\in{E^{[<\infty]}}$, there is an $\langle X\rangle\in\LF$ such that $X/\vec{x}\preceq X_{\vec{x}}$ for all $\vec{x}\sqsubseteq X$.
 	\end{enumerate}
 	In (a) and (c), the $X$ described is called a \emph{diagonalization} of $(X_n)$ or $(X_{\vec{x}})$, respectively. If $\LF$ is both full and a (strong) $(p)$-filter, then we refer to it as a \emph{(strong) $(p^+)$-filter}.
\end{defn}

Much of \cite{MR3864398} is devoted to showing that filters with these properties ``localize'' a Ramsey-theoretic dichotomy for block sequences in $E$, due to Rosendal \cite{MR2604856}, and one in Banach spaces, due to Gowers \cite{MR1954235}. Such dichotomies are phrased in terms of games:

Given $X\in{E^{[\infty]}}$, the \emph{asymptotic game} played below $X$, $F[X]$, is a two player game where the players alternate, with I going first and playing natural numbers $n_k$, and II responding with nonzero vectors $y_k$ forming a block sequence and such that $n_k<\min(\supp(y_k))$. 
\[
	\begin{matrix}
		\text{I}  & n_0 &       & n_1 &       & n_2 &       &\cdots\\
		\text{II} &	   & y_0 &       & y_1 &       & y_2 &\cdots
	\end{matrix}
\]

Likewise, the \emph{Gowers game} played below $X$, $G[X]$, is defined with I going first and playing infinite block sequences $Y_k\preceq X$, and II responding with nonzero vectors $y_k$ forming a block sequence and such that $y_k\in\langle Y_k\rangle$. 
\[
	\begin{matrix}
		\text{I}  & Y_0 &       & Y_1 &       & Y_2 &       &\cdots\\
		\text{II} &	   & y_0 &       & y_1 &       & y_2 &\cdots
	\end{matrix}
\]


In both of these games, the \emph{outcome} of a round of the game is the block sequence $(y_k)$ consisting of II's moves. The notion of a \emph{strategy} for one of the players is defined in a natural way. Given a set $\A\subseteq E^{[\infty]}$, we say that a player has a strategy for \emph{playing into (or out of)} $\A$ if they posses a strategy such that all the resulting outcomes line in (or out of) $\A$.

The following is the local form of Rosendal's dichotomy; Rosendal's original result can be recovered by simply omitting any mention of $\LF$.

\begin{thm}[Theorem 1.1 in \cite{MR3864398}]\label{thm:local_Rosendal}
	Let $\LF$ be a $(p^+)$-filter on $E$.\footnote{As mentioned in \cite{MR3864398}, an apparent weakening of the $(p)$-property akin to semiselectivity, namely that a sequence of dense open subsets of $\LF$ must possess a diagonalization in $\LF$, is all that is used in the proof of this result. However, it will be a consequence of Theorem \ref{thm:local_Rosendal_imp_p} below that, for block filters, this is equivalent to the $(p)$-property.} If $\A\subseteq{E^{[\infty]}}$ is analytic, then there is an $\langle X\rangle\in\LF$ such that either
	\begin{enumerate}[label=\textup{(\roman*)}]
		\item I has a strategy in $F[X]$ for playing out of $\A$, or
		\item II has a strategy in $G[X]$ for playing into $\A$.	
	\end{enumerate}
\end{thm}

While we won't deal explicitly with Banach spaces here, the spread condition, along with being $(p^+)$, was used to obtain the local form of Gowers result for Banach spaces (Theorem 1.4 in \cite{MR3864398}). These results are analogous to the way selective ultrafilters on $\omega$ and stable ordered-union ultrafilters on $\FIN$ localize the respective dichotomies for analytic partitions of $[\omega]^\omega$ and $\FIN^{[\infty]}$ (see \cite{MR0491197} and \cite{MR891244}). 

One apparent deficiency in Theorem \ref{thm:local_Rosendal} is that it is not obvious whether either conclusion, (i) or (ii), guarantees that $\LF$ meets the complement of $\A$ or $\A$ itself, respectively. This is rectified by the following assumption: 

\begin{defn}
	A block filter $\LF$ on $E$ is \emph{strategic}	if whenever $\alpha$ is a strategy for II in $G[X]$, where $\langle X\rangle\in\LF$, there is an outcome $Y$ of $\alpha$ such that $\langle Y\rangle\in\LF$.
\end{defn}

\begin{figure}
\[
	\xymatrix{\text{complete combinatorics}\ar@2{<->}[d]^{\text{ Theorem 1.2}}\\
		\text{strategic $(p^+)$-filter}\ar@2[d]^{\text{ Proposition 4.6}}\\
		\text{strong $(p^+)$-filter}\ar[d]^{\text{ Lemma 8.13}}\\
		\text{spread $(p^+)$-filter}\ar[d]\\
		\text{$(p^+)$-filter}\ar[d]^{\text{ Theorem 1.1}}\\
		\text{local Rosendal dichotomy}\ar[d]^{\text{ Proposition 3.6}} \\
		\text{full block filter}}
\]
\caption{The implications for block filters proved in \cite{MR3864398}.}
\label{fig:old_imps}
\end{figure}
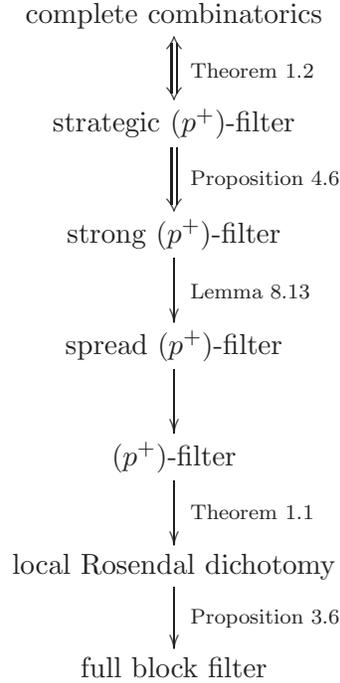

Under large cardinal assumptions, if $\LF$ is a strategic $(p^+)$-filter, then Theorem \ref{thm:local_Rosendal} can be extended to all ``reasonably definable'' subsets $\A$ and moreover, being ``strategic $(p^+)$'' exactly characterizes genericity of a block filter over the inner-model $\L(\R)$ (Theorems 1.2 and 1.3 in \cite{MR3864398}). This latter property is known in the literature as having \emph{complete combinatorics}.\footnote{Originally, ``complete combinatorics'' was used to describe genericity over $\mathrm{HOD}(\R)^{\V[G]}$, where $G$ is $\V$-generic for the Le\'vy collapse of a Mahlo cardinal \cite{MR996504}, a property proved (implicitly) for selective ultrafilters in \cite{MR0491197}. The contemporary usage avoids passing to a L\'evy collapse extension at the expense of stronger large cardinal hypotheses. The $\ZFC$ content of ``complete combinatorics'', in all examples of which the author is aware, is that the filter meets all dense open analytic sets, in the relevant $\sigma$-distributive Suslin partial order. That this characterizes strategic $(p^+)$-filters in $\ZFC$ is implicit in \cite{MR3864398}.}

Figure \ref{fig:old_imps} shows the implications between these various properties of block filters, as proved in \cite{MR3864398}, with references to the relevant propositions therein (the implication from ``spread $(p^+)$'' to ``$(p^+)$'' is trivial). The unidirectional double arrow $\Longrightarrow$ indicates a strict implication; when $|F|>2$, it is consistent with $\ZFC$ that there is a strong $(p^+)$-filter which is not strategic (this is Corollary 8.9 in the recent preprint \cite{Smythe:ParaRamseyBlockI}).

The goal of the present article is to investigate the combinatorics of block filters having the above properties, as well the possibility of reversing the remaining arrows in Figure \ref{fig:old_imps}. We begin by considering the special case of the finite field of order $2$ and its relationship to $\FIN$, where a complete analysis is possible. In general, we will see that the cardinality of the field $F$, in so far as it is either $2$, finite, or infinite, plays an important role. Those interested in spoilers may skip ahead to Figure \ref{fig:new_imps}. We also prove alternate characterizations of strong $(p^+)$-filters (Theorem \ref{thm:strong_p_restr_game}) and strategic $(p^+)$-filters (Theorem \ref{thm:+_strat}) using a restricted version of the Gowers game.

The results below are inspired by the various equivalent characterizations of selectivity for ultrafilters, originally proved by Booth and Kunen in \cite{MR0277371} (see Chapter 11 of \cite{MR3751612} for a modern treatment), and of stability for ordered-union ultrafilters, proved by Blass in \cite{MR891244}. Some of our result originate in the author's PhD thesis \cite{Smythe_thesis}, but have remained otherwise unpublished, while others are making their first appearance here.

\section{$\FIN$ and the finite field of order two}

A sequence $(a_n)$ in $\FIN$ is a \emph{block sequence} if $\max(a_n)<\min(a_{n+1})$ for all $n$. The set of all infinite block sequences in $\FIN$ is denoted by $\FIN^{[\infty]}$ and inherits a Polish topology from $\FIN^\omega$. We write $\FIN^{[<\infty]}$ for the set of finite block sequences in $\FIN$. Given $A,B\in\FIN^{[\infty]}$, we write $\langle A\rangle$ for the set of all finite unions of entries from $A$ and $A\preceq B$ if $\langle A\rangle\subseteq\langle B\rangle$ (likewise for $A/n$ and $A\preceq^* B$) to agree with our notation above. We reserve the notation $\langle A\rangle$ for when $A$ is a block sequence. For $A\in\FIN^{[\infty]}$, we write $A^{[\infty]}$ for the set of those $B\in\FIN^{[\infty]}$ such that $B\preceq A$.

The Ramsey theory for $\FIN$ is largely a consequence of the finite-unions form of Hindman's Theorem \cite{MR0349574}: 

\begin{thm}
	For any $C\subseteq\FIN$, there is an $A\in\FIN^{[\infty]}$ such that either $\langle A\rangle\subseteq C$ or $\langle A\rangle\cap C=\emptyset$.	
\end{thm}


The relevant notions for ultrafilters on $\FIN$ were defined by Blass in \cite{MR891244}: An ultrafilter $\LF$ of subsets of $\FIN$ is \emph{ordered-union} if it has a base of sets of the form $\langle A\rangle$. $\LF$ is \emph{stable} if whenever $\langle A_n\rangle\in\LF$ for all $n\in\omega$, there is a $\langle B\rangle\in\LF$ such that $B\preceq^* A_n$ for all $n$. These are connected to block filters on $E$ via the support map:

\begin{thm}[Theorem 6.3 in \cite{MR3864398}]\label{thm:supp_ordered_union}
	If $\LF$ is a full block filter on $E$, then
	\[
		\supp(\LF)=\{A\subseteq\FIN:\exists X\in\LF(A\supseteq\{\supp(v):v\in X\})\}
	\]	
	is an ordered-union ultrafilter on $\FIN$. If, moreover, $\LF$ is a $(p)$-filter, then $\supp(\LF)$ is stable.
\end{thm}

In the case when $|F|=2$, nonzero vectors in $E$ can be identified with their supports in $\FIN$, addition of vectors with disjoint supports corresponds to their union, and scalar multiplication trivializes. Thus, the study of block sequences in $E$ reduces to the study of block sequences in $\FIN$, and $\LF$ is a full ($(p^+)$, respectively) block filter if and only if it is an (stable) ordered-union ultrafilter: One direction is Theorem \ref{thm:supp_ordered_union}, while the converse follows from the fact that if $D\subseteq E$ is $\LU$-dense and $\LU$ is an \emph{ultra}filter, then $D\in\LU$.

We will see in Theorem \ref{thm:local_Rosendal_imp_p} below that the second to last implication in Figure \ref{fig:old_imps} reverses: For block filters, being a $(p^+)$-filter is equivalent to satisfying Theorem \ref{thm:local_Rosendal_imp_p}, regardless of the field. The $|F|=2$ case highlights a difficulty in understanding the last implication in Figure \ref{fig:old_imps}; whether it reverses in this case is equivalent to whether every ordered-union ultrafilter is stable, a long-standing open problem (see, e.g., \cite{MR2926315}). We will not attempt to shed any additional light on this question here.

While Theorem \ref{thm:local_Rosendal} can be rephrased for stable ordered-union ultrafilters and $\FIN$, a stronger result holds; stable ordered-union ultrafilters localize the infinite-dimensional form of Hindman's Theorem \cite{MR0349574} due to Milliken and Taylor \cite{MR0373906} \cite{MR424571}. This is one of several equivalents to stability proved in \cite{MR891244}:

\begin{thm}[Theorem 4.2 in \cite{MR891244}]\label{thm:local_Milliken_Taylor}
	Let $\LU$ be an ordered-union ultrafilter on $\FIN$. The following are equivalent:
	\begin{enumerate}[label=\textup{(\roman*)}]
		\item $\LU$ is stable.
		\item For any analytic set $\A\subseteq\FIN^{[\infty]}$, there is an $\langle A\rangle\in\LU$ such that either $A^{[\infty]}\subseteq\A$ or $A^{[\infty]}\cap\A=\emptyset$.
	\end{enumerate}
\end{thm}

Assuming large cardinal hypotheses, the methods of \cite{MR1644345} can be used to extend (ii) to all subsets $\A$ in $\L(\R)$ and prove complete combinatorics for stable ordered-union ultrafilters (see also the discussion on p.~121-122 of \cite{MR891244}). Consequently, when $|F|=2$, all but possibly the last of the conditions in Figure \ref{fig:old_imps} are equivalent. To see this directly:

\begin{cor}\label{cor:stable_ordered_strat}
	If $\LU$ is a stable ordered-union ultrafilter, then $\LU$ is strategic.	
\end{cor}

\begin{proof}
	Let $\langle A\rangle\in\LU$ and $\alpha$ be a strategy for II in $G[A]$.	By Lemma 4.7 in \cite{MR3864398}, there is an analytic set $\A$ of outcomes of $\alpha$ which is dense below $A$, in the sense of forcing with $(\FIN^{[\infty]},\preceq)$. By Theorem \ref{thm:local_Milliken_Taylor} applied to $A^{[\infty]}$, there is a $B\preceq A$ with $\langle B\rangle\in\LU$ such that either $B^{[\infty]}\subseteq\A$ or $B^{[\infty]}\cap \A=\emptyset$. Since $\A$ is dense below $B$, the latter is impossible. In particular, $B\in \A\subseteq[\alpha]$, so $\LU$ contains an outcome of $\alpha$. Thus, $\LU$ is strategic.
\end{proof}

Another variation on stability appears in the literature \cite{MR2330595} \cite{MR3717938}: An ultrafilter $\LU$ on $\FIN$ is \emph{selective} if it is ordered-union and whenever $\langle A_a\rangle\in\LU$ for all $a\in\FIN^{[<\infty]}$, there is a $\langle B\rangle\in\LU$ such that $B/a\preceq A_a$ for all $a\preceq B$. Note the resemblance to our strong $(p)$-property. While this property appears stronger than stability, it is, again, equivalent:

\begin{cor}
	If $\LU$ is a stable ordered-union ultrafilter on $\FIN$, then $\LU$ is selective.
\end{cor}

\begin{proof}
	Suppose we are given $\langle A_a\rangle\in\LU$ for all $a\in\FIN^{[<\infty]}$. Define
	\begin{align*}
		\mathbb{D}_0&=\{B\in\FIN^{[\infty]}: \text{ $B/a\preceq A_a$ for all $a\preceq B$}\}\\
		\mathbb{D}_1&=\{B\in\FIN^{[\infty]}: \text{ $\langle B\rangle$ and the $\langle A_a\rangle$'s do not generate a filter}\}.
	\end{align*}
	Let $\mathbb{D}=\mathbb{D}_0\cup\mathbb{D}_1$. It is straightforward to verify that $\mathbb{D}$ is analytic and dense open in $(\FIN^{[\infty]},\preceq)$. By Theorem \ref{thm:local_Milliken_Taylor}, there is a $\langle B\rangle\in\LU$ such that $B\in\mathbb{D}$. Clearly, $B\notin\mathbb{D}_1$, so $B\in\mathbb{D}_0$ and thus witnesses selectivity.
\end{proof}

Both of the previous corollaries are instances of complete combinatorics at work in the $\ZFC$ context.

Returning to the setting of an arbitrary countable field, we have seen that every $(p^+)$-filter produces a stable ordered-union ultrafilter. In Theorem \ref{thm:ordered_union_to_full}, we prove a converse. We'll need some notation: For $X=(x_n)\in{E^{[\infty]}}$, let $\supp(X)=(\supp(x_n))\in\FIN^{[\infty]}$. Part (a) of the following lemma implies that $\supp:{E^{[\infty]}}\to\FIN^{[\infty]}$ is a projection, in the sense of forcing.

\begin{lemma}\label{lem:supp_proj}
	\begin{enumerate}
		\item Suppose that $X\in{E^{[\infty]}}$ and $A\in\FIN^{[\infty]}$ are such that $A\preceq\supp(X)$. Then, there is a $Y\in{E^{[\infty]}}$ such that $Y\preceq X$ and $\supp(Y)=A$.
		\item Suppose that $(X_n)$ is a $\preceq^*$-decreasing sequence in ${E^{[\infty]}}$ and $A\in\FIN^{[\infty]}$ is such that $A\preceq^*\supp(X_n)$ for all $n$. Then, there is a $Y\in{E^{[\infty]}}$ such that $Y\preceq^* X_n$ for all $n$ and $\supp(Y)=A$.	
	\end{enumerate}
\end{lemma}

\begin{proof}
	Part (a) is easier, so we will just prove (b) here instead: Write $A=(a_k)$, with each $a_k\in\FIN$. For notational convenience, let $X_{-1}=(e_n)$ and $m_{-1}=-1$. For each $n\geq 0$, let $m_n$ be such that $A/m_n\preceq \supp(X_n)$ and $X_n/m_n\preceq X_{n-1}$. We may assume that each $m_n=\max(\supp(a_i))$ for some $i$ and that $m_n<m_{n+1}$. For each $n\geq -1$, and each of the finitely many $a_k$'s with $\supp(a_k)\subseteq(m_{n},m_{n+1}]$, choose $y_k\in\langle X_n\rangle$ such that $\supp(a_k)=y_k$. Let $Y=(y_k)$. Clearly, $\supp(Y)=A$. Moreover, for each $n$ and all $k$ with $\supp(y_k)\geq m_n$, $y_k\in\langle X_n/m_n\rangle$, and so $Y\preceq^*X_n$.
\end{proof}

\begin{lemma}\label{lem:U_dense_FIN}
	Let $\LF$ be a block filter on $E$ such that $\supp(\LF)$ is a stable ordered-union ultrafilter. If $D\subseteq E$ is $\LF$-dense and $\langle Y\rangle\in\LF$, then there is a $Z\preceq Y$ such that $\langle Z\rangle\subseteq D$ and $\supp(Z)\in\supp(\LF)$.
\end{lemma}

\begin{proof}
	Let
	\[
		\D=\{A\in\FIN^{[\infty]}:\exists Z\in{E^{[\infty]}}(Z\preceq Y \land \supp(Z)=A \land \langle Z\rangle\subseteq D)\},
	\]	
	an analytic subset of $\FIN^{[\infty]}$. By Theorem \ref{thm:local_Milliken_Taylor}, there is a $\langle B\rangle\in\supp(\LF)$ with $B\in\FIN^{[\infty]}$ such that either $B^{[\infty]}\subseteq\D$ or $B^{[\infty]}\cap\D=\emptyset$. We claim the latter cannot happen: As $\langle B\rangle\in\supp(\LF)$, there is a $\langle Y'\rangle\in\LF$ such that $\supp(Y')\preceq B$. Since $\LF$ is a block filter, we may further assume that $Y'\preceq Y$. As $D$ is $\LF$-dense, there is a $V\preceq Y'$ such that $\langle V\rangle\subseteq D$, and so $\supp(V)\in B^{[\infty]}\cap\D$. Thus, $B^{[\infty]}\subseteq\D$, and in particular, $B\in\D$. Any $Z\in{E^{[\infty]}}$ which witnesses $B\in\D$ will satisfy the desired conclusion.
\end{proof}

\begin{thm}\label{thm:ordered_union_to_full}
	$(\CH)$\footnote{$\CH$ is only used here in so far as it allows us to avoid diagonalizing uncountable-length sequences in $\LU$. If, instead, $\MA$ holds and $\LU$ was closed under diagonalizations of length $<2^{\aleph_0}$ (such stable ordered-union ultrafilters can be easily constructed using $\MA$ and Lemma 5 of \cite{MR2145797}), then our proof would go through \textit{mutatis mutandis}.} If $\LU$ is a stable ordered-union ultrafilter on $\FIN$, then there is a $(p^+)$-filter $\LF$ on $E$ such that $\supp(\LF)=\LU$.
\end{thm}

\begin{proof}
	Using $\CH$, we can enumerate all subsets of $E$ as $D_\xi$, and all elements $A\in\FIN^{[\infty]}$ such that $\langle A\rangle\in\LU$ as $A_\eta$, for $\xi,\eta<\aleph_1$. We will construct, via transfinite recursion, a $\preceq^*$-decreasing sequence $(X_\alpha)_{\alpha<\aleph_1}$ in ${E^{[\infty]}}$ that will generate the promised $(p^+)$-filter $\LF$.
	
	$\alpha=0$: Choose $X_0'\in{E^{[\infty]}}$ such that $\supp(X_0')=A_0$. If $D_0$ is such that there is some $Y\preceq X_0'$ with $\langle Y\rangle\subseteq D_0$ and $\langle\supp(Y)\rangle\in\LU$, then choose $X_0$ to be such a $Y$. If not, take $X_0=X_0'$.
	
	$\alpha=\beta+1$: Suppose we have defined $X_\gamma$ for $\gamma\leq\beta$ such that $\langle\supp(X_\gamma)\rangle\in\LU$. There is some $\langle B\rangle\in\LU$ such that $\langle B\rangle\subseteq \langle\supp(X_\beta)\rangle\cap\langle A_{\beta+1}\rangle$. Apply Lemma \ref{lem:supp_proj}(a) to obtain an $X_{\beta+1}'\in{E^{[\infty]}}$ such that $X_{\beta+1}'\preceq X_\beta$ and $\supp(X_{\beta+1}')=B$. If $D_{\beta+1}$ is such that there is some $Y\preceq X_{\beta+1}'$ with $\langle Y\rangle\subseteq D_{\beta+1}$ and $\langle\supp(Y)\rangle\in\LU$, then choose $X_{\beta+1}$ to be such a $Y$. If not, take $X_{\beta+1}=X_{\beta+1}'$.
	
	$\alpha=\beta$ for limit $\beta$: Suppose we have defined $X_\gamma$ for $\gamma<\beta$ such that $\langle\supp(X_\gamma)\rangle\in\LU$. Let $(\gamma_n)$ be a strictly increasing cofinal sequence in $\beta$. Since $\LU$ is stable, there is some $A\in\FIN^{[\infty]}$ such that $\langle A\rangle\in\LU$ and $A\preceq^* \supp(X_{\gamma_n})$ for all $n$. We may, moreover, assume that $A\preceq A_\beta$. By Lemma \ref{lem:supp_proj}(b), there is an $X_\beta'\in{E^{[\infty]}}$ such that $X_\beta'\preceq^* X_n$ for all $n$ and $\supp(X_\beta')=A$. If $D_{\beta}$ is such that there is some $Y\preceq X_\beta'$ with $\langle Y\rangle\subseteq D_{\beta}$ and $\langle\supp(Y)\rangle\in\LU$, then choose $X_{\beta}$ to be such a $Y$. If not, take $X_{\beta}=X_\beta'$. This completes the construction.
	
	Let $\LF$ be the block filter on $E$ generated by the $X_\alpha$'s. Our construction has ensured that $\LF$ has the $(p)$-property and that $\supp(\LF)\supseteq\LU$, and hence $\supp(\LF)=\LU$, since $\LU$ is an ultrafilter. It remains to verify that $\LF$ is full. Suppose that $D=D_\xi$ is $\LF$-dense. By Lemma \ref{lem:U_dense_FIN} applied to the $X_\xi'$ (for which we've ensured $\langle X_\xi'\rangle\in\LF$) found in stage $\xi$ of the above construction, there must be some $Y\preceq X_\xi'$ such that $\langle Y\rangle\subseteq D$ and $\langle\supp(Y)\rangle\in\LU$, meaning that $X_\xi$ was chosen so that $\langle X_\xi\rangle\subseteq D$. Thus, $\LF$ is a $(p^+)$-filter.
\end{proof}

It was shown in \cite{MR891244} that if $\LU$ is an ordered-union ultrafilter, then
\begin{align*}
	\min(\LU)&=\{\{\min(\supp(a)):a\in A\}:A\in\LU\}\\
	\max(\LU)&=\{\{\max(\supp(a)):a\in A\}:A\in\LU\}
\end{align*}
are nonisomorphic selective ultrafilters on $\omega$, and conversely, if $\LV_0$ and $\LV_1$ are nonisomorphic selective ultrafilters, then (assuming $\CH$) there is a stable ordered-union ultrafilter $\LU$ on $\FIN$ such that $\min(\LU)=\LV_0$ and $\max(\LU)=\LV_1$. This can now be combined with the previous theorem to get a similar conclusion for $(p^+)$-filters on $E$.

\section{Fullness and maximality}

A full block filter $\LF$ on $E$ is always maximal amongst block filters, and in fact is maximal with respect to all filters generated by infinite-dimensional subspaces of $E$. That is, for any infinite-dimensional subspace $V$ of $E$, if $V\cap X\neq\{0\}$ for all $X\in\LF$, then $V\in\LF$ (to see this, just let $D=V$ in the definition of ``full''). Filters of subspaces with the latter property were studied by Bergman and Hrushovski in \cite{MR1661256}, where they were called \emph{linear ultrafilters}; we will instead call them \emph{subspace maximal}. 

\begin{prop}
	Let $\LF$ be a filter generated by infinite-dimensional subspaces of $E$. The following are equivalent:
	\begin{enumerate}[label=\textup{(\roman*)}]
		\item $\LF$ is subspace maximal.
		\item For every subspace $V\subseteq E$, either $V\in\LF$ or there is some direct complement $V'$ of $V$ (i.e., $V\cap V'=\{0\}$ and $V\oplus V'=E$) such that $V'\in\LF$.
		\item For every linear transformation $T$ on $E$ (to any $F$-vector space), either $\ker(T)\in\LF$ or there is a subspace $X\in\LF$ such that $T\restr X$ is injective.
	\end{enumerate}
\end{prop}

\begin{proof}
	The equivalence if (i) and (ii) is part of Lemma 3 of \cite{MR1661256}. 
	
	(i $\Rightarrow$ iii) Given $T$, suppose that $T$ is not injective on any subspace $Y\in\LF$. This means that $\ker(T)$ has nontrivial intersection with every such $Y$. Hence, by subspace maximality, $\ker(T)\in\LF$.
	
	(iii $\Rightarrow$ i) Let $Y$ be an infinite-dimensional subspace of $E$ which has nontrivial intersection with every subspace $X\in\LF$. Let $Y'$ be a direct complement to $Y$ in $E$. Take $T:E\to E$ to be the unique linear transformation determined by 
	\[
		T(y+y')=y'
	\]
	for $y\in Y$ and $y'\in Y'$. So, $\ker(T)=Y$. If there was a subspace $X\in\LF$ such that $T\restr X$ was injective, then by assumption, $X\cap Y$ is nontrivial and $T\restr X\cap Y$ is injective, a contradiction. Thus, $Y=\ker(T)\in\LF$.
\end{proof}

We mention here a result from \cite{MR1661256} about the relationship between selective ultrafilters on $\omega$ and filters of subspaces of $E$: Proposition 18 in \cite{MR1661256} says that, given an ultrafilter $\LU$ on $\omega$, the set $\{\langle (e_i)_{i\in A}\rangle:A\in\LU\}$, together with the finite-codimensional subspaces of $E$, generates (via finite intersections and supersets) a subspace maximal filter on $E$ if and only if $\LU$ is selective. However, it is not clear if the resulting filter on $E$ can be a block filter. Moreover, as it is consistent with $\ZFC$ that there is a unique (up to isomorphism) selective ultrafilter, and hence no ordered-union ultrafilters (cf.~VI.5 in \cite{MR1623206} and the comments at the end of the previous section), one cannot obtain a full block filter on $E$ from a selective ultrafilter alone.


In contrast to the above forms of maximality, unless $|F|=2$, a block filter on $E$ is \emph{never} an ultrafilter (of sub\textit{sets}). This is a consequence of the existence of asymptotic pairs:

\begin{defn}
	\begin{enumerate}
		\item A set $A\subseteq E\setminus\{0\}$ is \emph{asymptotic} if for every infinite-dimensional subspace $V$ of $E$, $V\cap A\neq\emptyset$.
		\item An \emph{asymptotic pair} is a pair of disjoint asymptotic sets.	
	\end{enumerate}
\end{defn}

A standard construction of an asymptotic pair uses the \emph{oscillation} of a nonzero vector $v=\sum a_ne_n$, defined by
	\[
		\osc(v)=|\{i\in\supp(v):a_i\neq a_{i+1}\}|.
	\]	
	It is shown in the proof of Theorem 7 in \cite{MR2737185} that if $|F|>2$, then on any infinite-dimensional subspace of $E$, the range of $\osc$ contains arbitrarily long intervals (i.e., is a \emph{thick} set), and thus the sets
	\begin{align*}
		A_0 &= \{v\in E\setminus\{0\}:\osc(v)\text{ is even}\}\\
		A_1 &= \{v\in E\setminus\{0\}:\osc(v)\text{ is odd}\}
	\end{align*}
	form an asymptotic pair. Note that $\osc$, and thus the $A_i$, are invariant under multiplication by nonzero scalars.

Given a block filter $\LF$, a set $D\subseteq E$ is $\LF$-dense if and only if $A=E\setminus D$ fails to be asymptotic below every $\langle X\rangle\in\LF$. This immediately implies the following alternate characterization of fullness:

\begin{prop}
	A block filter $\LF$ on $E$ is full if and only if for every $A\subseteq E\setminus\{0\}$, there is an $\langle X\rangle\in\LF$ such that either $\langle X\rangle\cap A=\emptyset$ or $A$ is asymptotic below $\langle X\rangle$. \qed
\end{prop}

%

\section{The $(p)$-property and its relatives}

We begin this section by showing that if a block filter witnesses the local form of Rosendal's dichotomy, then it must have the $(p)$-property.

\begin{thm}\label{thm:local_Rosendal_imp_p}
	Let $\LF$ be a block filter on $E$. If all clopen subsets of ${E^{[\infty]}}$ satisfy the conclusion of Theorem \ref{thm:local_Rosendal_imp_p}, then $\LF$ has the $(p)$-property.
\end{thm}

\begin{proof}
	Let $\langle X_n\rangle\in\LF$ for each $n$. Define
	\[
		\A=\{(x_n)\in{E^{[\infty]}}:\text{ if $m\leq\max(\supp(x_0))$, then $x_1\in X_m$}\}.
	\]
	Clearly, $\A$ is a clopen subset of ${E^{[\infty]}}$. By our assumption, applied to $\A^c$, there is an $\langle X\rangle\in\LF$ such that either (i) I has a strategy in $F[X]$ for playing into $\A$ or (ii) II has a strategy in $G[X]$ for playing into $\A^c$.
	
	We claim that (ii) cannot happen. Suppose otherwise, denote II's strategy by $\alpha$, and consider the following round of $G[X]$: In the first inning, I plays $X$ and II responds with $\alpha(X)$. In the second inning, I plays some $Y\preceq X$ such that $\langle Y\rangle\subseteq\bigcap_{m\leq\max(\supp(\alpha(X)))}\langle X_m\rangle$,\footnote{Note that here, we could take $\langle Y\rangle\in\LF$. This shows that block filters witnessing Theorem \ref{thm:restr_local_Rosendal} below, while not necessarily full, must still have the $(p)$-property.} which defeats any possible next move by II, contrary to what we know about $\alpha$.
	
	Thus, (i) holds. Denote by $\sigma$ the resulting strategy for I. Let $Y=X/\sigma(\emptyset)$, so $\langle Y\rangle\in\LF$. Let $m$ be given. In the first inning of $F[X]$, let I play $\sigma(\emptyset)$, and let II play any $y\in \langle Y\rangle$ such that $m\leq\max(\supp(y))$. In the second inning, I plays $\sigma(y)$, which ensures that for any $z\in\langle Y/\sigma(y)\rangle$, $z\in\langle X_m\rangle$. In other words, $Y/\sigma(y)\preceq X_m$. Since $m$ was arbitrary, this shows that $Y\preceq^* X_m$ for all $m$, verifying the $(p)$-property.
\end{proof}

Next, we show that the $(p)$-property implies something which resembles the strong $(p)$-property, except that the family of elements of the filter which we diagonalize is indexed by finite sequences in $\FIN$ instead of in $E$.

\begin{thm}\label{thm:FIN_diag_p+}
	Let $\LF$ be a $(p^+)$-filter on $E$. Then, whenever $(\langle X_{\vec{a}}\rangle)_{\vec{a}\in\FIN^{[<\infty]}}$ is contained in $\LF$, there is an $\langle X\rangle\in\LF$ such that $X/\vec{a}\preceq X_{\vec{a}}$ whenever $\vec{a}\sqsubseteq\supp(X)$.
\end{thm}

\begin{proof}
	Let $(\langle X_{\vec{a}}\rangle)_{\vec{a}\in\FIN^{[<\infty]}}$ be given as described. Since $\FIN^{[<\infty]}$ is countable and $\LF$ is a $(p)$-filter, there is an $\langle X\rangle\in\LF$ such that $X\preceq^* X_{\vec{a}}$ for all $\vec{a}\in\FIN^{[<\infty]}$. Writing $\supp(X)^{[<\infty]}$ for those finite block sequences in $\FIN$ coming from $\langle X\rangle$, let
	\[
		\LB=\{\vec{a}\concat b\in \supp(X)^{[<\infty]}:\forall v\in\langle X\rangle(\supp(v)=b\rightarrow v\in \bigcap\{\langle X_{\vec{c}}\rangle:\vec{c}\sqsubseteq\vec{a}\})\}
	\]
	and
	\[
		\B=\{A\in \supp(X)^{[\infty]}:\forall n(A\restr n\in\LB)\}.
	\]
	Clearly, $\B$ is a Borel subset of $\supp(X)^{[\infty]}$. By Theorem \ref{thm:local_Milliken_Taylor} applied to the stable ordered-union ultrafilter (by Theorem \ref{thm:supp_ordered_union}) $\supp(\LF)$, there is $\langle Y\rangle\in\LF$ with $Y=(y_n)\preceq X$, such that either $\supp(Y)^{[\infty]}\subseteq\B$ or $\supp(Y)^{[\infty]}\cap\B=\emptyset$. Note, however, that the latter is impossible: Since $Y\preceq^* X_{\vec{a}}$ for all $\vec{a}\in\FIN^{[<\infty]}$, we can thin $Y=(y_n)$ out to a subsequence $Y'=(y_{n_k})$ such that $\supp(Y')\in\B$: take
	\begin{align*}
		y_{n_0}&\in\langle X_\emptyset\rangle\\
		y_{n_1}&\in\langle X_\emptyset\rangle\cap\langle X_{(\supp(y_{n_0}))}\rangle\\
		y_{n_2}&\in\langle X_{\emptyset}\rangle\cap\langle X_{(\supp(y_{n_0}))}\rangle\cap\langle X_{(\supp(y_{n_0}),\supp(y_{n_1}))}\rangle
	\end{align*}
	and so on. Thus, $\supp(Y)^{[\infty]}\subseteq\B$, and in particular, $\supp(Y)\in\B$, so $Y/\vec{a}\preceq X_{\vec{a}}$ whenever $\vec{a}\sqsubseteq \supp(Y)$.
\end{proof}

%
%

\begin{cor}\label{cor:p+_diag_max_supp}
	Let $\LF$ be a $(p^+)$-filter on $E$. Then, whenever $\langle X_n\rangle$ is in $\LF$ for all $n$, there is an $\langle X\rangle\in\LF$, with $X=(x_n)$, such that $X/x_n\preceq X_{\max(\supp(x_n))}$ for all $n$.\end{cor}

\begin{proof}
	Given $\langle X_n\rangle$ as described, let $X_{\vec{a}}=X_{\max(\vec{a})}$ for all $\vec{a}\in\FIN^{[<\infty]}$ and apply Theorem \ref{thm:FIN_diag_p+}.
\end{proof}

\begin{cor}\label{cor:p+_spread}
	Every $(p^+)$-filter on $E$ is spread.	
\end{cor}

\begin{proof}
	Let $\LF$ be a $(p^+)$-filter and $I_0<I_1<\cdots$ be an increasing sequence of nonempty intervals in $\omega$. Let $X=(e_n)$. For each $k\in\omega$, let $m_k$ be the least integer such that $k\leq\max(I_{m_k})$ and let $X_k=X/\max(I_{m_k+1})$. Let $Y=(y_n)$, with $\langle Y\rangle\in\LF$, be as in Corollary \ref{cor:p+_diag_max_supp}. We may assume $Y\preceq X/\max(I_0)$. Then, for any $n$, if $k=\max(\supp(y_n))$, then $Y/k\preceq X/\max(I_{m_k+1})$, and so $I_0<\supp(y_n)<I_{m_k+1}<\supp(y_{n+1})$.
\end{proof}

When $F$ is a finite field, we can go one step further:

\begin{cor}\label{cor:finite_p+_strong}
	Assume $|F|<\infty$. Every $(p^+)$-filter on $E$ is a strong $(p^+)$-filter.	
\end{cor}

\begin{proof}
	Let $\LF$ be a $(p^+)$-filter and $(\langle X_{\vec{x}}\rangle)_{\vec{x}\in{E^{[<\infty]}}}$ in $\LF$. Note that since $|F|<\infty$, for each $a\in\FIN$, there are only finitely many vectors $v\in E$ having support contained in $a$. For each $\vec{a}\in\FIN^{[<\infty]}$, let $\langle X_{\vec{a}}\rangle\in\LF$ be such that
	\[
		\langle X_{\vec{a}}\rangle \subseteq\bigcap\{\langle X_{\vec{x}}\rangle:\supp(\vec{x})\sqsubseteq\vec{a}\}.
	\]
	By Theorem \ref{thm:FIN_diag_p+}, there is a $\langle X\rangle\in\LF$ such that $X/\vec{a}\preceq X_{\vec{a}}$ for all $\vec{a}\sqsubseteq\supp(X)$. So, if $\vec{x}\sqsubseteq X$, then 
	\[
		X/\vec{x}=X/\supp(\vec{x})\preceq X_{\supp(\vec{x})}\preceq X_{\vec{x}}.
	\]
	This verifies the strong $(p)$-property.
\end{proof}

We do not know if Corollary \ref{cor:finite_p+_strong} holds for infinite fields.

We note here that the spread condition is analogous to another property for ultrafilters on $\omega$: Recall that an ultrafilter $\LU$ on $\omega$ is a \emph{q-point} if for every partition $\bigcup_m I_m$ of $\omega$ into finite sets, there exists an $x\in\LU$ such that $\forall m(|x\cap I_m|\leq 1)$. It is well-known that every selective ultrafilter is a q-point, though the converse (consistently) fails. Let's say (temporarily) that an ultrafilter $\LU$ on $\omega$ is \emph{spread} if for every sequence of finite intervals $I_0<I_1<I_2<\cdots$ in $\omega$, there exists an $x\in\LU$ such that for every $n$, there is an $m$ such that $I_0<x_n<I_m<x_{n+1}$, where $(x_n)$ is the increasing enumeration of $x$.

\begin{prop}\label{ref:prop_spread_q}
	Let $\LU$ be an ultrafilter on $\omega$. The following are equivalent:
	\begin{enumerate}[label=\rm{(\roman*)}]
		\item $\LU$ is a q-point
		\item For every sequence of finite sets $I_0<I_1<I_2<\cdots$ in $\N$, there exists a $x\in\LU$ such that $\forall m(|x\cap I_m|\leq 1)$
		\item $\LU$ is spread.
	\end{enumerate}
\end{prop}

\begin{proof}
	(i $\Rightarrow$ ii): This is trivial.
	
	(ii $\Rightarrow$ iii): Let $I_0<I_1<I_2<\cdots$ be a sequence of intervals in $\omega$. Let $x\in\LU$ be as in (ii). We may assume that $I_0<x_0$. We partition $x=u\cup v$ as follows: $u_n=x_{2n}$ and $v_n=x_{2n+1}$ for all $n$, where $(x_n)$ is the increasing enumeration of $y$.  For every $n$, since $u_n=x_{2n}$, $v_n=x_{2n+1}$, and $u_{n+1}=x_{2n+2}$ must be contained in three distinct $I_k$'s, the middle interval must separate $u_n$ and $u_{n+1}$, that is, there is an $m$ such that $I_0<u_n<I_m<u_{n+1}$. Similarly for the $v_n$. Since $\LU$ is an ultrafilter, one of $u$ or $v$ must be in $\LU$.
	
	(iii $\Rightarrow$ i): Let $\bigcup_m I_m$ be a partition of $\omega$ into finite sets. We define an interval partition $\omega=\bigcup_k J_k$ as follows: $J_0=[0,\max I_0]$. Let $J_1$ be the smallest interval immediately above $J_0$ such that $J_0\cup J_1$ covers $I_1$ and all $I_m$ for which $I_m\cap J_0\neq\emptyset$. Continue in this fashion, letting $J_{k+1}$ be the smallest interval immediately above $J_k$ such that $J_0\cup\cdots\cup J_k\cup J_{k+1}$ covers $I_{k+1}$ and all $I_m$ for which $I_m\cap(J_0\cup\cdots\cup J_k)\neq\emptyset$. Let $x\in\LU$ be as in the definition of spread applied to $J_0<J_1<\cdots$. Towards a contradiction, suppose that $x_i<x_j$ are both in some $I_m$. Let $n$ be the least such that $I_m\subseteq J_0\cup\cdots\cup J_n$. We may assume $n>1$ (otherwise, we are done). By minimality of $n$, $I_m\cap(J_0\cup\cdots\cup J_{n-2})=\emptyset$. Thus, $I_m\subseteq J_{n-1}\cup J_n$. But then, $x_i$ and $x_j$ fail to be separated by one of the $J_k$'s, contrary to $x$ witnessing that $\LU$ is spread.
\end{proof}

\section{The restricted Gowers game and strategic filters}

Given a block filter $\LF$ and $\langle X\rangle\in\LF$, we define the \emph{restricted Gowers game} $G_\LF[X]$ below $X$ exactly like $G[X]$ except that player I is restricted to playing $Y\preceq X$ such that $\langle Y\rangle\in\LF$.  Since all subspaces spanned by tails of $X$ are automatically in $\LF$, we may think of $G_\LF[X]$ as an intermediate between the games $F[X]$ and $G[X]$. Throughout this section, we will say that an outcome $Y$ of one of the games is ``in $\LF$'' if $\langle Y\rangle$ is. Our first result here relates strategies for I in $G_\LF[X]$ to the strong $(p)$-property, and is based on a characterization of selective ultrafilters (Theorem 11.17(b) in \cite{MR3751612}).

\begin{thm}\label{thm:strong_p_restr_game}	
	Let $\LF$ be a block filter on $E$. $\LF$ has the strong $(p)$-property if and only if for every $X\in\LF$ and every strategy $\sigma$ for I in $G_{\LF}[X]$, there is an outcome of $\sigma$ in $\LF$.
\end{thm}

\begin{proof}	
	($\Rightarrow$) Towards a contradiction, suppose that $\sigma$ is a strategy for I in $G_{\LF}[X]$, $\langle X\rangle\in\LF$, and no outcome of $\sigma$ is in $\LF$. Define sets $\LA_{\vec{x}}\subseteq\LF$ as follows: $\LA_\emptyset=\{\langle\sigma(\emptyset)\rangle\}$ and in general, $\LA_{\vec{x}}$ is the set of all $\langle Y\rangle\in\LF$ such that $Y$ is played by I, when I follows $\sigma$ and $\vec{x}=(x_0,\ldots,x_{n-1})$ are the first $n$ moves by II. Some $\vec{x}$ may not be valid moves for II against $\sigma$, in which case we let $\LA_{\vec{x}}=\LA_{\vec{x}'}$ where $\vec{x}'$ is the maximal initial segment of $\vec{x}$ consisting of valid moves. Then, for all $\vec{x}$, $\LA_{\vec{x}}$ is finite, and $\LA_{\vec{x}}\subseteq\LA_{\vec{y}}$ whenever $\vec{x}\sqsubseteq\vec{y}$.
	
	For each $\vec{x}$, pick $\langle Y_{\vec{x}}\rangle\in\LF$ such that for all $Y\in\LA_{\vec{x}}$, $Y_{\vec{x}}\preceq Y$. By the strong $(p)$-property, there is a $\langle Y\rangle\in\LF$, say $Y=(y_n)\preceq X$, such that $Y/\vec{y}\preceq Y_{\vec{y}}$ for all $\vec{y}\sqsubseteq Y$.
	
	Consider the play of $G_{\LF}[X]$ wherein I follows $\sigma$ and II plays $y_0$, $y_1$, etc. This is a valid play by II by our choice of $Y$: $y_0\in \langle Y_\emptyset\rangle\subseteq\langle\sigma(\emptyset)\rangle$, $y_1\in \langle Y/(y_0)\rangle\subseteq\langle Y_{(y_0)}\rangle\subseteq\langle\sigma(y_0)\rangle$, etc. The resulting outcome is $Y$, and $\langle Y\rangle\in\LF$, a contradiction to our assumption about $\sigma$.
	
	($\Leftarrow$) Suppose that $\LF$ does not have the strong $(p)$-property, so there are $\langle X_{\vec{x}}\rangle\in\LF$ for all $\vec{x}\in{E^{[<\infty]}}$ such there for no $\langle X\rangle\in\LF$ is it the case that $X/\vec{x}\preceq X_{\vec{x}}$ for all $\vec{x}\sqsubseteq X$. Take $\langle X\rangle\in\LF$ arbitrary. We define a strategy $\sigma$ for I in $G_{\LF}[X]$ as follows: Start by playing $Y_\emptyset\preceq X,X_\emptyset$. If II plays $y_0\in\langle Y_0\rangle$, respond by playing some $Y_{(y_0)}\preceq Y_\emptyset,X_{(y_0)}$. In general, if II has played $(y_0,\ldots,y_k)$, respond by playing some $Y_{(y_0,\ldots,y_k)}\preceq Y_{(y_0,\ldots,y_{k-1})},X_{(y_0,\ldots,y_k)}$. Note that in each move, we can always find such a $\langle Y_{\vec{y}}\rangle\in\LF$ since $\LF$ is a block filter. If $Y$ is an outcome of a round of $G_{\LF}[X]$ where I followed $\sigma$, then for every $\vec{y}\sqsubseteq Y$, $Y/\vec{y}\preceq X_{\vec{y}}$. In other words, $Y$ is a diagonalization of $\langle X_{\vec{x}}\rangle_{\vec{x}\in{E^{[<\infty]}}}$, and thus by assumption, cannot be in $\LF$.
\end{proof}

Since every strategy for I in $F[X]$ is also a strategy for I in $G_\LF[X]$, it follows that if $\LF$ is a strong $(p)$-filter, $\langle X\rangle\in\LF$, and $\sigma$ a strategy for I in $F[X]$, then there is an outcome of $\sigma$ in $\LF$ (this is Theorem 4.3 in \cite{MR3864398}).

The restricted Gowers game can be used to prove a version of Theorem \ref{thm:local_Rosendal} for $(p)$-filters without the extra assumption of fullness. This result is due independently to the author and, in more generality, to No\'e de Rancourt:

\begin{thm}[Theorem 3.11.5 in \cite{Smythe_thesis} and Theorem 3.3 in \cite{MR4127870}]\label{thm:restr_local_Rosendal}
	Let $\LF$ be a $(p)$-filter on $E$. If $\A\subseteq{E^{[\infty]}}$ is analytic, then there is an $\langle X\rangle\in\LF$ such that either
	\begin{enumerate}[label=\textup{(\roman*)}]
		\item I has a strategy in $F[X]$ for playing out of $\A$, or
		\item II has a strategy in $G_\LF[X]$ for playing into $\A$.	
	\end{enumerate}	
\end{thm}

The following is a version of being strategic for the restricted games.

\begin{defn}
	A block filter $\LF$ on $E$ is \emph{$+$-strategic} if whenever $\alpha$ is a strategy for II in $G_{\LF}[X]$, where $\langle X\rangle\in\LF$, there is an outcome of $\alpha$ which is in $\LF$.	
\end{defn}

What is the difference between being $+$-strategic and strategic? We will see below that, at least for $(p)$-filters, it is exactly fullness.


We will need the following notion and a lemma: A \emph{tree} is a subset $T\subseteq{E^{[<\infty]}}$ which is closed under initial segments. The set $[T]$ of infinite branches through $T$ is a closed subset of ${E^{[\infty]}}$.

\begin{lemma}[cf.~Lemma 6.4 in \cite{MR2523338}]\label{lem:H_dense_tree}
	Let $\LF$ be a filter on $E$, $\langle X\rangle\in\LF$, and $\alpha$ a strategy for II in $G_{\LF}[X]$. Then, there is a tree $T\subseteq{E^{[<\infty]}}$ such that:
	\begin{enumerate}[label=\textup{(\roman*)}] 
		\item $[T]\subseteq[\alpha]$, and
		\item whenever $(y_0,\ldots,y_n)\in T$ and $\langle Y\rangle\in\LF$, there is a $y\in\langle Y\rangle$ so that $(y_0,\ldots,y_n,y)\in T$. 
	\end{enumerate}
\end{lemma}

\begin{proof}
	We will define a pair of trees $T\subseteq{E^{[<\infty]}}$ and $S\subseteq\LF^{<\infty}$ as follows: Put $\emptyset\in T$ and $S$. The first level of $T$ consists of all $(y)\in{E^{[<\infty]}}$ such that $y$ is a ``first move'' by II according to $\alpha$. That is, there some $Y\preceq X$ such that $\langle Y\rangle\in\LF$ and $\alpha(Y)=y$. For each such $y$, pick a corresponding $Y$ in its preimage under $\alpha$; these comprise the first level of $S$.
	
	We continue inductively. Having put $(y_0,\ldots,y_n)\in T$ and $(Y_0,\ldots, Y_n)\in S$ with $\alpha(Y_0,\ldots,Y_i)=y_i$ for $i\leq n$, we put $(y_0,\ldots,y_n,y)$ if there is some $Y\preceq X$ such that $\langle Y\rangle\in\LF$ and $\alpha(Y_0,\ldots,Y_n,Y)=y$. Choose some $Y$ with this property and put $(Y_0,\ldots, Y_n,Y)$ into $S$. 
	
	Clearly, $[T]\subseteq[\alpha]$. To see that $[T]$ satisfies (ii), let $(y_0,\ldots, y_n)\in T$ and $\langle Y\rangle\in\LF$ with $Y\preceq X$. Let $(Y_0,\ldots,Y_n)\in S$ be such that $\alpha(Y_0,\ldots,Y_i)=y_i$ for $i\leq n$, and put $y_{n+1}=\alpha(Y_0,\ldots,Y_n,Y)\in\langle Y\rangle$. By construction, there is some $Y_{n+1}$ with $(Y_0,\ldots,Y_n,Y_{n+1})\in S$ and $y_{n+1}=\alpha(Y_0,\ldots,Y_n,Y_{n+1})$.
\end{proof}


\begin{thm}\label{thm:+_strat}
	Let $\LF$ be $(p)$-filter on $E$. Then, $\LF$ is $+$-strategic if and only if $\LF$ is strategic and full.
\end{thm}

\begin{proof}
	($\Rightarrow$) First observe that +-strategic implies strategic: Given any $\langle X\rangle\in\LF$ and strategy $\alpha$ for II in $G[X]$, let $\alpha'$ be the restriction of $\alpha$ to $G_{\LF}[X]$ (in the obvious sense). Since $\LF$ is +-strategic, there is an outcome of $\alpha'$, and thus of $\alpha$, in $\LF$.
	
	To see that $\LF$ is full, let $D\subseteq E$ be $\LF$-dense and put
	\[
		\D=\{Y\in{E^{[\infty]}}:\langle Y\rangle\subseteq D\},
	\]
	a closed subset of ${E^{[\infty]}}$. By Theorem \ref{thm:restr_local_Rosendal}, there is an $\langle X\rangle\in\LF$ such that either I has a strategy in $F[X]$ for playing into $\D^c$, or II has a strategy in $G_{\LF}[X]$ for playing in $\D$. However, the former is impossible: pick $Z\preceq X$ in $\D$ and let II in $F[X]$ always play elements of $\langle Z\rangle$. As $\LF$ is +-strategic, there is some outcome of II's strategy in $G_{\LF}[Y]$ in $\LF$, verifying fullness.
	
	($\Leftarrow$) Assume that $\LF$ is strategic and full, that is, $\LF$ is a strategic $(p^+)$-filter. We must prove that $\LF$ is +-strategic. Let $\langle X\rangle\in\LF$ and $\alpha$ a strategy for II in $G_{\LF}[X]$. Let $T\subseteq{E^{[<\infty]}}$ be as in Lemma \ref{lem:H_dense_tree}. By Theorem \ref{thm:local_Rosendal}, there is a $Y\preceq X$ such that $\langle Y\rangle\in\LF$ and either I has a strategy in $F[Y]$ for playing into $[T]^c$, or II has a strategy in $G[Y]$ for playing into $[T]$. The former is impossible as II has a strategy in $F[Y]$ for playing into $[T]$: Inductively apply the property in Lemma \ref{lem:H_dense_tree}(ii) to the tail block sequences played by I in $F[Y]$. Thus, II has a strategy in $G[Y]$ for playing into $[T]$. As $\LF$ is strategic, there is some outcome of this strategy, and thus some element of $[\alpha]$, in $\LF$.
\end{proof}

Theorem \ref{thm:+_strat} is a crucial part of the consistency proof of the existence of a strong $(p^+)$-filter which is not strategic in \cite{Smythe:ParaRamseyBlockI}.

\section{Summary and further questions}\label{sec:end}

Figure \ref{fig:new_imps} shows where the implications between the properties described in the Introduction stand at the end of this article. The single arrows $\rightarrow$ indicate that the converse remains open for arbitrary fields (with $\ZFC$ as a base theory). In addition to sorting out the remaining implications in this diagram, there are a few other questions we wish to highlight for further investigation (each when $|F|>2$):

\begin{ques}
	Does the Continuum Hypothesis (or Martin's Axiom) imply the existence of $(p^+)$-filters which are not strategic?
\end{ques}

\begin{ques}
	Is it consistent with $\ZFC$ that there are stable ordered-union ultrafilters on $\FIN$, but not $(p^+)$-filters on $E$?
\end{ques}

\begin{ques}
	Is there a meaningful version of the Rudin--Keisler ordering and its accompanying theory for filters on vector spaces? If so, are $(p^+)$-filters minimal?
\end{ques}

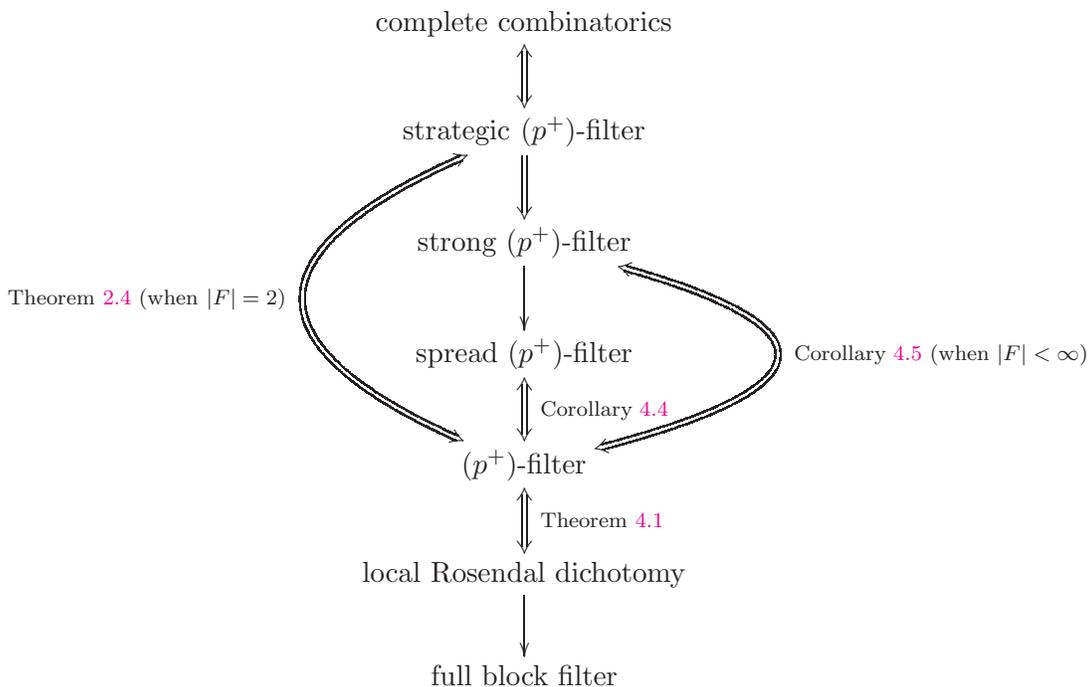
\begin{figure}
\[
	\hspace{-2em}
	\xymatrix{\text{complete combinatorics}\ar@2{<->}[d]\\
		\text{strategic $(p^+)$-filter}\ar@2[d]\\
		\text{strong $(p^+)$-filter}\ar[d]\\
		\text{spread $(p^+)$-filter}\ar@2{<->}[d]^{\text{ Corollary \ref{cor:p+_spread}}}\\
		\text{$(p^+)$-filter}\ar@2{<->}[d]^{\text{ Theorem \ref{thm:local_Rosendal_imp_p}}} \ar@2@{<->}@/^7pc/[uuu]^{\text{Theorem \ref{cor:stable_ordered_strat} (when $|F|=2$) }} \ar@2@{<->}@/_8pc/[uu]_{\text{ Corollary \ref{cor:finite_p+_strong} (when $|F|<\infty$)}}\\
		\text{local Rosendal dichotomy}\ar[d]\\
		\text{full block filter}}
\]
\caption{Updated implications from Figure \ref{fig:old_imps}.}
\label{fig:new_imps}
\end{figure}

\bibliography{../math_bib}{}

\begin{thebibliography}{10}

\bibitem{MR1661256}
G.~M. Bergman and E.~Hrushovski.
\newblock Linear ultrafilters.
\newblock {\em Comm. Algebra}, 26(12):4079--4113, 1998.

\bibitem{MR891244}
A.~Blass.
\newblock Ultrafilters related to {H}indman's finite-unions theorem and its
  extensions.
\newblock In {\em Logic and combinatorics ({A}rcata, {C}alif., 1985)},
  volume~65 of {\em Contemp. Math.}, pages 89--124. Amer. Math. Soc.,
  Providence, RI, 1987.

\bibitem{MR0277371}
D.~Booth.
\newblock Ultrafilters on a countable set.
\newblock {\em Ann. Math. Logic}, 2(1):1--24, 1970/1971.

\bibitem{MR4127870}
N.~de~Rancourt.
\newblock Ramsey theory without pigeonhole principle and the adversarial
  {R}amsey principle.
\newblock {\em Trans. Amer. Math. Soc.}, 373(7):5025--5056, 2020.

\bibitem{MR1644345}
I.~Farah.
\newblock Semiselective coideals.
\newblock {\em Mathematika}, 45(1):79--103, 1998.

\bibitem{MR2145797}
V.~Ferenczi and C.~Rosendal.
\newblock Ergodic {B}anach spaces.
\newblock {\em Adv. Math.}, 195(1):259--282, 2005.

\bibitem{MR2523338}
V.~Ferenczi and C.~Rosendal.
\newblock Banach spaces without minimal subspaces.
\newblock {\em J. Funct. Anal.}, 257(1):149--193, 2009.

\bibitem{MR1954235}
W.~T. Gowers.
\newblock An infinite {R}amsey theorem and some {B}anach-space dichotomies.
\newblock {\em Ann. of Math. (2)}, 156(3):797--833, 2002.

\bibitem{MR3751612}
L.~J. Halbeisen.
\newblock {\em Combinatorial set theory}.
\newblock Springer Monographs in Mathematics. Springer, Cham, 2017.
\newblock With a gentle introduction to forcing, Second edition.

\bibitem{MR0349574}
N.~Hindman.
\newblock Finite sums from sequences within cells of a partition of
  {$\mathbb{N}$}.
\newblock {\em J. Combinatorial Theory Ser. A}, 17:1--11, 1974.

\bibitem{MR2926315}
P.~Krautzberger.
\newblock On union ultrafilters.
\newblock {\em Order}, 29(2):317--343, 2012.

\bibitem{MR996504}
C.~Laflamme.
\newblock Forcing with filters and complete combinatorics.
\newblock {\em Ann. Pure Appl. Logic}, 42(2):125--163, 1989.

\bibitem{MR2737185}
C.~Laflamme, L.~Nguyen Van~Th\'e, M.~Pouzet, and N.~Sauer.
\newblock Partitions and indivisibility properties of countable dimensional
  vector spaces.
\newblock {\em J. Combin. Theory Ser. A}, 118(1):67--77, 2011.

\bibitem{MR0491197}
A.~R.~D. Mathias.
\newblock Happy families.
\newblock {\em Ann. Math. Logic}, 12(1):59--111, 1977.

\bibitem{MR2330595}
J.~G. Mijares.
\newblock A notion of selective ultrafilter corresponding to topological
  {R}amsey spaces.
\newblock {\em MLQ Math. Log. Q.}, 53(3):255--267, 2007.

\bibitem{MR0373906}
K.~R. Milliken.
\newblock Ramsey's theorem with sums or unions.
\newblock {\em J. Combinatorial Theory Ser. A}, 18:276--290, 1975.

\bibitem{MR2604856}
C.~Rosendal.
\newblock An exact {R}amsey principle for block sequences.
\newblock {\em Collect. Math.}, 61(1):25--36, 2010.

\bibitem{MR1623206}
S.~Shelah.
\newblock {\em Proper and improper forcing}.
\newblock Perspectives in Mathematical Logic. Springer-Verlag, Berlin, second
  edition, 1998.

\bibitem{Smythe:ParaRamseyBlockI}
I.~B. Smythe.
\newblock Parametrizing the {R}amsey theory of block sequences {I}: {D}iscrete
  vector spaces.
\newblock Preprint. [arXiv:2108.00544].

\bibitem{Smythe_thesis}
I.~B. Smythe.
\newblock {\em Set theory in infinite-dimensional vector spaces}.
\newblock PhD thesis, Cornell University, 2017.

\bibitem{MR3864398}
I.~B. Smythe.
\newblock A local {R}amsey theory for block sequences.
\newblock {\em Trans. Amer. Math. Soc.}, 370(12):8859--8893, 2018.

\bibitem{MR424571}
A.~D. Taylor.
\newblock A canonical partition relation for finite subsets of {$\omega $}.
\newblock {\em J. Combinatorial Theory Ser. A}, 21(2):137--146, 1976.

\bibitem{MR3717938}
Y.~Y. Zheng.
\newblock Selective ultrafilters on {FIN}.
\newblock {\em Proc. Amer. Math. Soc.}, 145(12):5071--5086, 2017.

\end{thebibliography}
\bibliographystyle{abbrv}

\end{document}